\documentclass{article}
\usepackage[affil-it]{authblk}

\usepackage{amssymb, amsfonts, amsthm, stmaryrd, xcolor ,amsmath}

\usepackage[inline]{enumitem}

\newtheorem{theorem}{Theorem}[section]

\newtheorem{corollary}[theorem]{Corollary}
\newtheorem{lemma}[theorem]{Lemma}

\theoremstyle{definition}
\newtheorem{definition}[theorem]{Definition}
\newtheorem{example}[theorem]{Example}
\newtheorem{remark}[theorem]{Remark}

\usepackage{multicol}
\newcommand{\sfiven}[1]{{\rm Alt}_{#1}}
\newcommand{\CS}{\Gamma}

\newcommand{\bne}{-}

\newcommand{\eqm}{\approx}
\newcommand{\I}{{\bf i}}
\newcommand{\C}{{\bf c}}
\newcommand{\Kur}{{\bf Kur}}
\newcommand{\baire}[1]{{\bf Baire}(#1)}
\newcommand{\bquot}[1]{\wp({#1}) / \meager }
\newcommand{\meager}{{\bf M}}

\newcommand{\power}[1]{2^{#1}}

\newcommand{\B}{\Box}
\newcommand{\D}{\Diamond}

\newcommand{\fr}{\mathfrak}
\newcommand{\lang}[2]{{\mathcal L}^{#1}_{#2}}
\newcommand{\lb}{\llbracket}
\newcommand{\rb}{\rrbracket}
\newcommand{\val}[1]{\lb {#1} \rb}

\begin{document}

\title{The Baire closure and its logic}

\author{Guram Bezhanishvili}
\affil{Department of Mathematical Sciences\\
New Mexico State University\\
1290 Frenger Mall\\
MSC 3MB / Science Hall 236\\
Las Cruces, New Mexico 88003-8001
USA\\
{\tt guram@nmsu.edu}}

\author{David Fernández-Duque}
\affil{Department of Philosophy\\
University of Barcelona\\
C/ de Montalegre 6-8\\
08001 Barcelona\\
Spain\\
{\tt fernandez-duque@ub.edu}}

\maketitle

\begin{abstract}
The Baire algebra of a topological space $X$ is the quotient of the algebra of all subsets of $X$ modulo the meager sets.
We show that this Boolean algebra can be endowed with a natural closure operator, resulting in a closure algebra which we denote $\baire X$.
We identify the modal logic of such algebras to be the well-known system $\sf S5$, and prove soundness and strong completeness for the cases where $X$ is crowded and either completely metrizable and continuum-sized or locally compact Hausdorff.
We also show that every extension of $\sf S5$ is the modal logic of a subalgebra of $\baire X$, and that soundness and strong completeness also holds in the language with the universal modality.
\end{abstract}

\section{Introduction}

Canonical examples of Boolean algebras include the powerset $\wp(X)$ of a set $X$, as well as its subalgebras; indeed, every Boolean algebra is isomorphic to a subalgebra of a powerset algebra 
(see, e.g., \cite[p.~28]{Kop89}). 
Alternately, one can consider quotients of the form $\wp(X)/ {\bf I}$, where $\bf I$ is a suitable ideal of $\wp(X)$.
Some of the most familiar such ideals are the ideal $\mathbf N$ of null sets when $X$ is a measure space 
(see, e.g., \cite[p.~233]{Kop89}), or the ideal $\meager$ of meager sets when $X$ is a topological space 
(see, e.g., \cite[p.~182]{Kop89}); recall that a set is {\em nowhere dense} if the interior of its closure is empty, and that it is {\em meager} if it is a countable union of nowhere dense sets.
The quotient $\wp (X) / \meager $ gives rise to the {\em Baire algebra} of $X$.

When $X$ is a topological space, the powerset algebra of $X$ comes equipped with the usual closure operator $\C \colon \wp(X) \to \wp(X)$.
This operator satisfies some familiar properties also known as the {\em Kuratowski axioms,} including e.g.~$A\subseteq \C A$ (see Section \ref{secPrel} for the full list), and any operator satisfying these axioms uniquely determines a topology on $X$.
One can more generally consider a Boolean algebra $B$ with an operator $\C \colon B \to B$ satisfying the same axioms; such algebras are the {\em closure algebras} of McKinsey and Tarski \cite{MT44}.
It is then a natural question to ask whether a closure operator on $\wp(X)$ carries over to quotients $\wp(X)/{\bf I}$ in a meaningful way.
This question has already been answered in the affirmative by Fern\'andez-Duque \cite{Fer10} and Lando \cite{Lan12} in the setting of the {\em Lebesgue measure algebra,} defined as the quotient of the Borel sets of reals modulo the 
null sets.
In an unpublished work, Bjorndahl has also considered validity modulo the nowhere dense sets, although he does not work directly with the algebraic quotient.
In this article, we will explore this question in the setting of the Baire algebra of a topological space $X$.

One subtlety when defining a closure operator for such quotients is that, denoting the equivalence class of $Y$ by $[Y]$, the definition $\C[Y] : = [\C Y]$ does not yield a well-defined operation: for example, in the above-mentioned quotients, $[\mathbb Q] = [\varnothing]$ (as $\mathbb Q$ is both meager and of measure zero), yet $[\C\mathbb Q] = [\mathbb R]\neq [\varnothing] = [\C \varnothing]$.
Instead, we must compute the closure of a set directly within the quotient, by the expression
\[ 
\C  a   = \inf \{ [C] :  a   \sqsubseteq [C] \text{ and }C\text{ is closed} \} , 
\]
where $\sqsubseteq$ denotes the partial order on the quotient algebra. We refer to the element $\C a$ as the {\em Baire closure} of $a$. 
We utilize a nontrivial fact in topology that the Boolean algebra of Borel sets of an arbitrary topological space modulo the ideal of meager sets is complete (see \cite[p.~75]{Sik60})
to show that this produces a closure operator on the Baire algebra of {\em any} topological space $X$, and we denote the resulting closure algebra by $\baire X $.
\color{black}

Much as Boolean algebras provide semantics for propositional logic, closure algebras provide semantics for {\em modal logic,} which extends propositional logic with an operator $\D  $ that we will interpret as a closure operator, along with its dual $\B$, interpreted as interior.
Such topological semantics of modal logic, as well as the closely related intuitionistic logic, predates their now-widespread relational semantics. For intuitionistic logic it was first developed by Stone \cite{Sto37b} and Tarski \cite{Tar38}, and for modal logic by Tsao-Chen \cite{TC38}, McKinsey \cite{McK41}, and McKinsey and Tarski \cite{MT44}.
Under this interpretation, the modal logic of all topological spaces turns out to be the well-known modal system of Lewis, $\sf S4$ (see Section \ref{secPrel} for the definition). Other well-known extensions of $\sf S4$ also turn out to be the modal logics of interesting topological spaces. To give a couple of examples:
\begin{itemize}
\item ${\sf S4.2}:= {\sf S4} + \Diamond\Box p\to\Box\Diamond p$ is the modal logic of all extremally disconnected spaces \cite{Gab01}.
\item ${\sf S4.3}:={\sf S4} + \Box(\Box p\to q)\vee\Box(\Box q\to p)$ is the modal logic of all hereditarily extremally disconnected spaces \cite{BBBM15}.
\end{itemize}
Here we recall that a space $X$ is extremally disconnected if the closure of each open set is open, and $X$ is hereditarily extremally disconnected if each subspace of $X$ is extremally disconnected.

These topological completeness results can be strengthened as follows.
By the celebrated McKinsey-Tarski theorem \cite{MT44}, $\sf S4$ is the modal logic of any crowded metrizable space.\footnote{We recall that a space $X$ is {\em crowded} or {\em dense-in-itself} if it has no isolated points. It is worth pointing out that the original McKinsey-Tarski theorem also had the separability (equivalently the second countability) assumption on $X$, which was later removed by Rasiowa and Sikorski \cite{RS63} by an elaborate use of the Axiom of Choice. For a modern proof of this result see \cite{BBBM17a}.} In fact, $\sf S4$ is strongly sound and complete with respect to any crowded metrizable space \cite{Kre13,GH17}. In particular, $\sf S4$ is the modal logic of the real unit interval $[0,1]$.
On the other hand, $\sf S4.2$ is the modal logic of the Gleason cover of $[0,1]$ (see \cite{BH12}), while $\sf S4.3$ is the modal logic of a countable (hereditarily) extremally disconnected subspace of the Gleason cover of $[0,1]$ (see \cite{BBBM15}).
In contrast, ${\sf S5}:={\sf S4} + \D p \to \B\D p $, which is one of the best known modal logics, is not complete with respect to any class of spaces that satisfy even weak separation axioms. Indeed, $\Diamond\Box p\to\Box p$ is valid in a topological space iff each open set is also closed. Thus, a $T_0$-space validates $\sf S5$ iff it is discrete.

One can also characterize the modal logics of closure algebras that are not based on a powerset.
Fern\'andez-Duque \cite{Fer10} and Lando \cite{Lan12} have shown that $\sf S4$ is the logic of the Lebesgue measure algebra, and in this article we will characterize the logic of Baire algebras.
An element $c\in \baire X$ is closed (resp.~open) if $c=[C]$ for some closed (resp.~open) $C$.
A distinguished feature of $\baire X$ is that $c\in \baire X$ is open iff it is closed. This yields that $\sf S5$ is sound with respect to $\baire X$ for any topological space $X$. For completeness,
we refine Hewitt's \cite{Hew43} well-known concept of {\em resolvability} to that of {\em Baire resolvability.}
A space is resolvable if it can be partitioned into two dense sets.
Similarly, a closure algebra $B$ is resolvable if there are $a,b$ that are {\em orthogonal} ($a\wedge b = 0$) and {\em dense} ($\C a = \C b = 1$).
If we denote $|\mathbb R|$ by $\mathfrak c$, we show that if $X$ is crowded and either a complete metric space of cardinality $\mathfrak c$ or a locally compact Hausdorff space, then $\baire X$ is resolvable.
In fact, such algebras are $\mathfrak c$-resolvable, meaning that we can find $\mathfrak c$-many dense and pairwise orthogonal elements of $\baire X$.
Our main tool in proving these results is the Disjoint Refinement Lemma (see, e.g., \cite[Lem.~7.5]{CN74}).

Using these resolvability results we show that if $X$ is a crowded, continuum-sized, complete metrizable space, then $\sf S5$ is strongly complete for $\baire X$, yielding a variant of the McKinsey-Tarski theorem for $\sf S5$.
In view of $\mathfrak c$-resolvability, strong completeness holds even if we extend the propositional language with continuum-many propositional variables.
This yields a sharper version of strong completeness than that given in the literature, where only countable languages are typically considered.
Some of our other main results include that every extension of $\sf S5$ is the modal logic of some subalgebra of $\baire X$ for any crowded second-countable completely metrizable space $X$. These results also hold if we replace completely metrizable by locally compact Hausdorff.
Finally, we show how to extend our results to the setting of the universal modality.

\section{Preliminaries}\label{secPrel}

We assume some basic familiarity with Boolean algebras, topological spaces, and ordinal and cardinal arithmetic 
(see, e.g.,~\cite{Kop89,Eng89,CN74}).
In this section we will review closure algebras and their relation to modal logic.

\subsection{Closure algebras}

For each topological space $X$, the powerset $\wp(X)$ is a Boolean algebra
and the closure operator $\C:\wp(X)\to\wp(X)$ satisfies the Kuratowski axioms: \[
A\subseteq\C A, \ \C\C A\subseteq \C A, \ \C(A \cup B)=\C A\cup\C B, \mbox{ and } \C \varnothing = \varnothing.
\]
We call the pair $(\wp(X),\C)$ the \emph{Kuratowski algebra} of $X$ and denote it by $\Kur(X)$.

McKinsey and Tarski \cite{MT44} generalized the concept of a closure on $\wp(X)$ to that of a closure on an arbitrary Boolean algebra.

\begin{definition}\label{closAlg}
A {\em closure algebra} is a pair $(B,\C)$, where $B$ is a Boolean algebra and $\C \colon B \to B$ satisfies, for every $a,b\in B$,
\begin{enumerate}
\item $
a\le\C a$,

\item $\C\C a\le\C a$,

\item $\C(a\vee b)=\C a\vee\C b, $ and

\item $\C 0=0$.

\end{enumerate}
Let $\I a:= {-}\C {-}a$.
If $\C$ moreover satisfies
\[\C a = \I \C a,\]
then $(B,\C)$ is a {\em monadic algebra.} 
\end{definition}

\begin{remark}
The notion of a monadic algebra is due to Halmos \cite{Hal56}.
\end{remark}

As is the case for Boolean algebras, closure algebras can be represented as subalgebras of algebras based on a powerset.

\begin{theorem} \cite[Thm.~2.4]{MT44}
Each closure algebra is isomorphic to a subalgebra of the Kuratowski algebra $\Kur(X)$ of some topological space $X$.
\end{theorem}

Recall that Boolean algebras are partial orders $(B,{\leq})$ with a least element and a greatest element, usually denoted by $0$ and $1$,
and where for $a,b\in B$ their meet (greatest lower bound) $a\wedge b$
and complement $-a$
are defined, and satisfy the usual axioms (see, e.g., \cite{Hal63,Sik60}).
If $S\subseteq B$, then $\bigwedge S$ denotes its meet
and $\bigvee S$ its join
when they exist; the algebra $B$ is {\em complete} provided these always exist.
The next definition goes back to Halmos \cite{Hal56}.

\begin{definition}
We say that $A \subseteq B$ is {\em relatively complete} in $B$ if for each $b\in B$, the set $
\{
a\in A \mid b\leq a
\}$
has a least element (which then is $\bigwedge\{a\in A\mid b\le a\}$).
\end{definition}

Of particular importance are countable joins and meets, also called {\em $\sigma$-joins} and {\em $\sigma$-meets;} the algebra $B$ is \emph{$\sigma$-complete} if such joins and meets always exist, in which case $B$ is a \emph{$\sigma$-algebra}.
An {\em ideal} is a set ${\bf I}\subseteq B$ which is closed under finite joins and which for each element $a\in \bf I$ also contains all elements underneath $a$; it is \emph{proper} if $1\notin \bf I$; and it is a {\em $\sigma$-ideal} if it is closed under $\sigma$-joins.
For an ideal $\bf I$, recall that $B/\bf I$ is the set of equivalence classes of $B$ under the equivalence relation given by $a\sim b$ if $a-b,b-a\in \bf I$. We order $B/\bf I$ by $[a] \leq [b]$ if $a-b\in \bf I$. The quotient $B/\bf I$ is then a Boolean algebra; moreover, if 
$B$ is $\sigma$-complete and $\bf I$ is a $\sigma$-ideal, then $B/\bf I$ is $\sigma$-complete (see, e.g., \cite[p.~56]{Hal63}).

\subsection{Modal logic}

We will work with the basic unimodal language as well as its extension with the universal modality.
However, as we are interested in strong completeness for possibly uncountable sets of formulas, we want to allow for an arbitrary number of propositional variables.
Given a cardinal $\lambda$, let $\mathbb P^\lambda = \{p_\iota \mid \iota < \lambda\}$ be a set of $\lambda$-many propositional variables. The {\em modal language} $\lang \lambda{\D\forall}$ is defined by the grammar (in Backus-Naur form)
\[\varphi,\psi := \  \  \  \  p \  \  \ | \  \  \ \varphi\wedge\psi \ \  \  | \  \  \ \neg \varphi \  \  \ | \  \  \ \D \varphi \  \  \ | \  \  \ \forall \varphi  \]
where $p\in \mathbb P^\lambda$. We also use standard shorthands, for example defining $\bot$ as $p\wedge \neg p$ (where $p$ is any fixed variable), $\varphi\vee\psi$ as $\neg(\neg\varphi\wedge\neg \psi)$, and $\B\varphi$ as $\neg\D \neg\varphi$. We denote the $\forall$-free fragment by $\lang\lambda\D$, and we omit the superscript $\lambda$ when $\lambda = \omega$.

We use $\D$ as primitive rather than $\B$ since, historically, algebraic semantics was presented in terms of closure-like operators on Boolean algebras.
For our purposes, a {\em logic} is understood as a set of formulas closed under modus ponens, substitution, and necessitation.
If $\Lambda$ is any logic over $\mathcal L$, then $\Lambda^\lambda$ denotes the least logic over $\mathcal L^\lambda$ containing $\Lambda \cap \mathcal L^\lambda$.

We restrict our attention to logics above $\sf S4$, and especially to $\sf S5$ and its extensions. 
A standard axiomatization of $\sf S4$ in $\lang\lambda\D$ (i.e., of ${\sf S4}^\lambda$) is given by all (classical) propositional tautologies and the axioms and rules 
\begin{multicols}2
\begin{description}

\item[$\rm M$] $\D ( p \vee q) \to \D p \vee \D q $

\vfill

\item[$\rm T$] $ p \to \D p $

\vfill

\item[$\rm 4$] $\D \D p \to \D p $

\vfill

\item[$\rm N$] $\neg \D \bot$
\columnbreak

\item[$\rm Sub$] $\dfrac{\varphi(p_1,\dots,p_n)}{\varphi(\psi_1,\dots,\psi_n)}$
\\\\

\item[$\rm MP$] $\dfrac{\varphi \ \ \varphi \to \psi}\psi$
\\\\

\item[$\rm Mon$] $\dfrac{ \varphi \to \psi}{\D \varphi \to \D \psi}$.

\end{description}
\end{multicols}

The logic ${\sf S5}^\lambda$ is then obtained by adding the axiom 
\begin{description}
\item[$\rm 5$] $\D p \to \B\D p$.
\end{description}

Given a logic $\Lambda$ over $\lang\lambda\D$, we obtain a new logic $\Lambda{\sf U}$ by adding the $\sf S5$ axioms and rules for $ \forall$ (or rather for $\exists:=\neg\forall\neg$) and the connecting axiom $\D\varphi \to \exists \varphi$. We will specifically be interested in $\sf S5U$.

The language $\lang\lambda\D$ has familiar Kripke semantics based on frames $\mathfrak F = (W,R)$, where $R \subseteq W\times W$ (see, e.g., \cite{CZ97,BRV01}).
We will not review this semantics in detail, and instead regard it as a special case of topological or, more generally, algebraic semantics.

Closure algebras provide the algebraic semantics for $\sf S4$ and its {\em normal} extensions, meaning those logics containing the axioms of $\sf S4$ and closed under the rules of $\sf S4$.

\begin{definition}\label{DefSem}
An {\em algebraic model} of $\sf S4U$ is a structure $\fr M=(B,\val\cdot),$ where $B$ is a closure algebra and $\lb\cdot\rb\colon\lang \lambda{\D\forall} \to B$ is a {\em valuation for $\lang \lambda{\D\forall}$;} that is, a function such that $\lb p\rb \in B$ for each $p \in \mathbb P^\lambda$ and
\[
\begin{array}{rclcrcl}
\lb\varphi\wedge\psi\rb &=&\lb\varphi\rb\wedge \lb\psi\rb& \ \ \ \ \ \ \ \ &
\lb\neg \varphi \rb &=& - \lb\varphi\rb\\
\val{\D\varphi}&=&\C \val\varphi&&
\val{\forall\varphi} &=&
\begin{cases}
1 & \text{if $\val \varphi = 1$}\\
0 & \text{otherwise.}
\end{cases}
\end{array}
\]
We may also say that $\fr M$ is an \emph{algebraic model of $\sf S4$,} and  
that the restriction of $\val \cdot $ to $\lang \lambda{\D }$ is a {\em valuation for $\lang \lambda{\D } $.}
\end{definition}

Validity, soundness, and completeness are then defined in the usual way:

\begin{definition}
Let $\fr M = (B,\val\cdot)$ be an algebraic model and $\Gamma\cup\{\varphi\}\subseteq\lang\lambda{\D\forall}$.
\begin{enumerate}
\item We say that $\varphi$ is {\em valid} in $\fr M$, written $\fr M\models\varphi$, if $\val\varphi  = 1$. If $\fr M\models\gamma$ for each $\gamma\in\Gamma$, then we say that $\Gamma$ is {\em valid} in $\fr M$ and write $\fr M\models\Gamma$.
\item We write $B \models \varphi$ if $(B,\val\cdot) \models \varphi$ for every valuation $\val\cdot$ on $B$, and define $B \models \Gamma$ similarly.
\item If $\Omega$ is a class of closure algebras, we say that $\varphi$ is {\em valid in $\Omega$} provided $B \models \varphi$ for each $B \in \Omega$. That $\Gamma$ is {\em valid in $\Omega$} is defined similarly.
\item For $S \subseteq \{\D,\forall\}$, we denote the set of valid $\lang \lambda{S}$-formulas in
$\Omega$ by ${\sf Log}_S(\Omega)$. If $\Omega = \{ B \}$, we simply write ${\sf Log}_S( B )$.
\item A logic $\Lambda$ is {\em sound} for $\Omega$ if $ \Lambda \subseteq {\sf Log}_S(\Omega)$, and {\em complete} if $ \Lambda \supseteq {\sf Log}_S(\Omega)$.
\end{enumerate}
\end{definition}

We will also be interested in {\em strong} completeness of logics.

\begin{definition}
Let $\varphi$ be a formula, $\Gamma$ a set of formulas, and $\Lambda$ a logic in $\lang\lambda{\D\forall}$.
\begin{enumerate}

\item We write $\Gamma\vdash_\Lambda\varphi$ if there is a finite $\Delta\subseteq\Gamma$ such that $(\forall\bigwedge\Delta) \to \varphi \in \Lambda$.
We omit $\Gamma$ when $\Gamma = \varnothing$.
If $\Lambda$ is a logic in $\lang\lambda{\D }$, $\forall\bigwedge\Delta$ should be replaced by $\Box \bigwedge\Delta$. 

\item For a closure algebra $B$, we write $\Gamma\models_B\varphi$ if $B\models \Gamma $ implies 
$B\models\varphi$.\footnote{This is the `global consequence relation', which should not be confused with
the `local consequence relation', defined pointwise on Kripke frames.
The former is more natural in algebraic semantics.}

\item We say that $\Lambda$ is {\em strongly complete} for a closure algebra $B$ if $\Gamma\models_B\varphi$ implies $\Gamma\vdash_\Lambda\varphi$,
and that $\Lambda$ is {\em strongly complete} for a class of structures $\Omega$ if $\Gamma\models_\Omega\varphi$ implies $\Gamma\vdash_\Lambda\varphi$.
\end{enumerate}
\end{definition}

We remark that strong completeness, as we defined it, does not imply soundness.
As we had mentioned previously, Kripke frames can be seen as a special case of Kuratowski algebras.
If $W$ is a set and $R\subseteq W\times W$, then $R$ induces an operator on $\wp(W)$ given by
\[
A\mapsto R^{-1}(A):=\{w\in W \mid wRa \mbox{ for some } a\in A\}.
\]
If $W$ is nonempty and $R$ reflexive and transitive, we say that $(W,R)$ is an {\em $\sf S4$-frame}. It is well known and not hard to check that in this case $(\wp(W),R^{-1} )$ is a Kuratowski algebra, and that it is the Kuratowski algebra of the topology on $W$ whose open sets are those $U\subseteq W$ that satisfy $R(U) \subseteq U$, recalling that
\[
R(U)=\{w\in W\mid uRw \mbox{ for some } u\in U\}.
\]
If in addition $R$ is symmetric, then we say that $(W,R)$ is an {\em $\sf S5$-frame.}
We will tacitly identify an $\sf S4$-frame with the associated topological space and even with the associated Kuratowski algebra.

It is well known that the topologies arising from $\sf S4$-frames
are {\em Alexandroff spaces}; that is, topologies where arbitrary intersections of open sets are open.
It is a consequence of the results of McKinsey and Tarski \cite{MT44} and Kripke \cite{Kri63} that $\sf S4$ is sound for the class of closure algebras, and complete for the class of $\sf S4$-frames.
It is well known that $\sf S4$ and $\sf S5$ are strongly complete, with respect to 
both global and local consequence relations;
both consequence relations are discussed in \cite{BRV01}; on the other hand, \cite{RS63} only focuses on the global consequence relation (using different terminology), which is the route we take here.
The literature typically considers countable languages, but the adaptation of these results to uncountable languages is straightforward; we provide a sketch below.

\begin{theorem}\label{theoS4comp}\label{theoS5comp}
For an infinite cardinal $\lambda$, we have:
\begin{enumerate}
\item[{\em (1)}] ${\sf S4}^\lambda$ and ${\sf S4U}^\lambda$ are strongly complete for the class of all $\sf S4$-frames of cardinality at most $\lambda$.
\item[{\em (2)}] ${\sf S5}^\lambda$ and ${\sf S5U}^\lambda$ are strongly complete for the class of all $\sf S5$-frames of cardinality at most $\lambda$.\color{black} 
\end{enumerate}
\end{theorem}

\begin{proof}[Proof sketch]
We outline how a standard proof, as given e.g.~in 
\cite[p.~203]{BRV01}, can be adapted to obtain the results as stated.
Let $\Lambda $ be one of the logics mentioned above, and suppose that $\Gamma\not\vdash_{\Lambda} \varphi$.
The completeness proof by the relativized canonical model $\mathfrak M^\Gamma_c = (W^\Gamma_c,R^\Gamma_c,\val\cdot^\Gamma_c)$, where $W^\Gamma_c$ is the set of maximal consistent sets containing $\Gamma$, carries over mostly verbatim for uncountable languages.
The relativized canonical model has the property that $\mathfrak  M^\Gamma_c \models \Gamma$ and $\mathfrak  M^\Gamma_c \not \models \varphi$.
The only caveat is that the proof of the Lindenbaum lemma, 
that any consistent set of formulas $\Delta$ can be extended to a maximal consistent set $\Delta' \supseteq \Delta$,
requires the use of Zorn's lemma.

The set of worlds $W^\Gamma_c$ may have cardinality greater than $\lambda$.
However, viewing $\lang\lambda{\D\forall}$ as a 
fragment of the first-order logic via the standard translation~\cite[p.~84]{BRV01}, we may apply the downward L\"owenheim-Skolem theorem to obtain a model $\mathfrak M$ of size at most $\lambda$ such that $\mathfrak  M  \models \Gamma$ and $\mathfrak  M  \not \models \varphi$, as required.
 \end{proof}

For a topological space $X$, it is well known (and not hard to see) that the following are equivalent:
\begin{enumerate}
\item $X\models \sf S5$;
\item every open set of $X$ is closed;
\item $X$ has a basis which is also a partition of $X$;
\item $X$ is an $\sf S5$-frame.
\end{enumerate}
Thus, unlike the case of $\sf S4$, no generality is achieved
by passing to topological semantics of $\sf S5$. 
As we will see, the story is quite different for algebraic semantics.

We conclude this preliminary section by recalling the extensions of $\sf S5$.
For each natural number $n>0$, let
\[
\sfiven n := \bigwedge_{i = 1}^{n+1} \D p_i \to \bigvee_{1\leq i<j\leq n+1} \D(p_i\wedge p_j).
\]
The logic ${\sf S5}_n$ is the extension of $\sf S5$ by the above axiom.
We note that an $\sf S5$-frame validates $\sfiven n$ iff each point has at most $n$ distinct successors, or `alternatives'. 

The next result is well known, and was originally proved by Scroggs \cite{Scr51}. 
We remark that in this article we use $\subset$ to denote strict inclusion.

\begin{theorem}[Scroggs' Theorem]\label{theoScrogg}
The consistent extensions of $\sf S5$ form the 
following $(\omega+1)$-chain $($with respect to $\supset):$
\[
{\sf S5}_1\supset\dots\supset{\sf S5}_n\supset\dots\supset{\sf S5}.
\]
\end{theorem}

Note that Scroggs' Theorem applies to $\mathcal L^\lambda_\D$ for any infinite $\lambda$ since any logic $\Lambda$ over $\mathcal L^\lambda_\D$ is uniquely determined by $ \Lambda \cap \mathcal L^\omega_\D$ (as logics are 
closed under substitution).
However, it is possible to exhibit infinite sets of formulas that may only be satisfied  on uncountable models. For example, to satisfy 
\[
\{\D p_\alpha \mid \alpha<\lambda\} \cup \{\B \neg(p_\alpha\wedge p_\beta ) \mid \alpha<\beta<\lambda \}
\]
requires one point for each $p_\alpha$.

\begin{definition}\label{defClusterSize}
Let $(W,R)$ be an $\sf S5$-frame.
Then it is a disjoint union of equivalence classes, or {\em clusters}, $\bigcup_{\iota < \kappa} C_\iota$, where each $C_\iota$ is of the form $R(\{w\})$ (which we henceforth write as $R(w)$) for some $w\in W$. We define:
\begin{itemize}

\item the {\em number of clusters} to be $\kappa$;

\item the {\em lower cluster size} to be $\min_{\iota<\kappa}|C_\iota|$, and

\item the {\em upper cluster size} to be $\sup_{\iota<\kappa} |C_\iota| $.

\end{itemize}
The unique (up to isomorphism) frame with one cluster of size $\lambda$ is the {\em $\lambda$-cluster,} and we denote it by $\mathfrak C_\lambda$.
\end{definition}

Theorem \ref{theoS4comp} can be sharpened using the structures $\mathfrak C_\lambda$.

\begin{theorem}\label{theoS5n}
For each finite $n>0$, ${\sf S5}_n$ is sound and strongly complete for the $n$-cluster, while for infinite $\lambda$, ${\sf S5}^\lambda$ is sound and strongly complete for any $\kappa$-cluster with $\kappa\ge\lambda$.
\end{theorem}

\begin{proof}[Proof sketch]
It follows from 
Scroggs' Theorem that ${\sf S5}_n$ is sound and complete for the $n$-cluster. 
Strong completeness follows from compactness of the first-order logic (applied to the standard translations of the theorems of ${\sf S5}_n$) and using the formula stating that there are at most $n$ elements.
For infinite $\lambda$, we have that ${\sf S5}^\lambda$ is sound and strongly complete for the class of $\sf S5$-frames.
So, if $\Gamma\not\vdash \varphi$, there is a model $\mathfrak M$ and a world $w$ of $\mathfrak M$ such that $\mathfrak M \models\Gamma$ but $ w \not \in \val \varphi$.
The submodel of $\mathfrak M$ generated by $w$ is the cluster $R(w)$ and satisfies the same formulas.
We may assume that $R(w)$ is infinite 
by adding duplicates of a point if needed, which does not affect modal formulas.
Then using the upwards or downwards L\"owenheim-Skolem theorem, we can assume 
that this cluster is isomorphic to $\mathfrak C_\kappa$.
\end{proof}

\section{Baire algebras as a new semantics for $\sf S5$}

In this section we introduce the main concept of the paper, that of Baire algebras, which provide a new semantics for $\sf S5$. Baire algebras are obtained from topological spaces $X$ by modding out $\wp(X)$ by the $\sigma$-ideal $\meager$ of meager subsets of $X$. We show that the closure operator $\C$ on $X$ gives rise to a closure operator on $\bquot X$, so that $(\bquot X,\C)$ is a monadic algebra.

We start by the following well-known definition.

\begin{definition}
Let $X$ be a topological space and $A \subseteq X$.
\begin{enumerate}
\item $A$ is {\em nowhere dense} if $\I\C A =  \varnothing$.
\item $A$ is {\em meager} if it is a $\sigma$-union of nowhere dense sets.
\item $\meager$ denotes the set of meager subsets of $X$.
\item $X$ is a {\em Baire space} if for each nonempty open subset $U$ of $X$, we have that $U \notin \meager$.
\end{enumerate}
\end{definition}

It is easy to see that $\meager$ is a $\sigma$-ideal of $\wp(X)$.
Therefore, the quotient $\bquot X$ is a $\sigma$-algebra (see, e.g., \cite[Sec.~13]{Hal63}). However, $\bquot X$ is not always a complete Boolean algebra (see, e.g., \cite[Sec.~25]{Hal63}).

We recall that elements of $\bquot X$ are equivalence classes of the equivalence relation given by
\[
A\eqm B \mbox{ iff } A \setminus B,B\setminus A\in\meager.
\]
We write $[A]$ for the equivalence class of $A$ under $\eqm$. Then $[A]\sqsubseteq [B]$ iff $A\setminus B\in\meager$, and the operations
on $\bquot X$ are given by
\begin{equation}\label{eqSigmaUnion}
 \bigsqcap_{n<\omega} [A_n] = [\bigcap_{n<\omega} A_n ] \mbox{ and } \bne [A] = [X\setminus A].
\end{equation}
On the other hand, as we already pointed out above, if $S\subseteq \bquot X$ is uncountable, then $\bigsqcap S$ may not always exist.

\begin{example}\label{examS4}\
\begin{enumerate}
\item Let $(W,R)$ be an $\sf S5$-frame, and identify it with the topological space whose opens are unions of equivalence classes of the equivalence relation $R$. For each $w\in W$, since $R(w)$ is the least open containing $w$, which is also closed, $\varnothing$ is the only nowhere dense subset of $W$. Therefore, ${\bf M}=\{\varnothing\}$, and so $\wp(W)/{\bf M}$ is isomorphic to $\wp(W)$. 
\item More generally, if $(W,R)$ is an $\sf S4$-frame, 
a point $w\in W$ is {\em quasi-maximal} if $wRv$ implies $vRw$. Let ${\sf qmax}W$ be the set of quasi-maximal points.
Then $w\in{\sf qmax}W$ iff $R(w) \subseteq R^{-1}(w)$. Therefore, 
$w\in{\sf qmax}W$ iff $\{w\}$ is not nowhere dense.  
Thus, if $W$ is countable, ${\bf M}=\{ A \subseteq W \mid A \subseteq W\setminus {\sf qmax}W \}$, 
and hence $\wp(W)/{\bf M}$ is isomorphic to $\wp({\sf qmax}W)$ (via  $[A]\mapsto A\cap {\sf qmax}W$).
\end{enumerate}
\end{example}

\begin{definition}\label{defClopen}
Let $X$ be a topological space.
\begin{enumerate}
\item We call $a\in\wp(X)/\meager$ \emph{open} if $a=[U]$ for some open $U\in\wp(X)$, and \emph{closed} if $a=[F]$ for some closed $F\in\wp(X)$.
An element that is both open and closed is \emph{clopen.}
\item We denote the set of closed elements of $\wp(X)/\meager$ by $\CS$.
\end{enumerate}
\end{definition}

It is easy to see that $\CS$ is a bounded sublattice of $\wp(X)/\meager$. We next show that $\CS$ is in fact a Boolean subalgebra of $\wp(X)/\meager$.
For this we point out that if $U$ is open, then $\C U\setminus U$ is nowhere dense (since it is closed and does not have any nonempty open subsets), so $[\C U] = [U]$. Similarly, if $F$ is closed, then $[F] = [\I F]$. 

\begin{theorem}\label{lemmHComplete}
For every topological space $X$, we have that $\CS$ is a Boolean subalgebra of $\bquot X$.
\end{theorem}

\proof
Since $\CS$ is a bounded sublattice of $\bquot X$,
it is sufficient to show that $a\in \CS$ implies $-a\in\CS$. From $a\in\CS$ it follows that $a=[F]$ for some $F$ closed in $X$. Then $-a=[X\setminus F]$. Since $X\setminus F$ is open,
we have that $X\setminus F \eqm \C(X\setminus F)$. Thus, $-a=[X\setminus F]=[\C(X\setminus F)]$, and hence $-a\in\CS$.
\endproof

In fact, $\Gamma$ is a $\sigma$-subalgebra of $\wp(X)/\meager$, which follows from Equation~(\ref{eqSigmaUnion}) and the fact that intersection of closed sets is closed. But we can say more.
For this we recall that a subset of $X$ is a {\em Borel set} if it belongs to the $\sigma$-algebra generated by the open sets. We write ${\bf Borel}(X)$ for the $\sigma$-algebra of Borel sets and ${\bf Borel}(X)$ for the $\sigma$-algebra of Borel sets and ${\bf Borel}(X)/\meager$ for the $\sigma$-subalgebra of $\wp(X)/\meager$ isomorphic to ${\bf Borel}(X)/(\meager\cap {\bf Borel}(X))$.

\begin{theorem} \label{thm: Gamma=Borel}
For every topological space $X$, we have that
$\Gamma={\bf Borel}(X)/{\bf M}$.
\end{theorem}

\begin{proof}
Since each closed set is Borel, we have $\Gamma\subseteq{\bf Borel}(X)/{\bf M}$. For the reverse inclusion, every open set $U$ is equivalent to its own closure, so $[U]\in\Gamma$. But $\Gamma$ is a $\sigma$-algebra, so we conclude that ${\bf Borel}(X)/{\bf M}\subseteq\Gamma$, hence the equality. 
\end{proof}

Since ${\bf Borel}(X)/{\bf M}$ is always a complete Boolean algebra (see \cite[p.~75]{Sik60}), we obtain:

\begin{corollary}
For every topological space $X$, we have that $\Gamma$ is a complete Boolean algebra.
\end{corollary}


In addition, we show that $\CS$ is a relatively complete subalgebra of $\wp(X)/\meager$: 

\begin{theorem}\label{lem:relatively complete}
For an arbitrary topological space $X$, we have that $\CS$ is a relatively complete subalgebra of $\wp(X)/\meager$.
\end{theorem}

\begin{proof}
Let $a = [A] \in \wp(X)/\meager $ and let $\mathcal C_a $ be the collection of all closed sets $C$ such that $[A]\sqsubseteq [C]$. By \cite[p.~75]{Sik60}, $\bigsqcap \{ [C] \mid C \in \mathcal C_a \} = [ \bigcap \mathcal C_a ]$. Thus, $\mathcal C_a$ has a least element, and hence $\CS$ is a relatively complete subalgebra of $\wp(X)/\meager$.
\end{proof}

\color{black}

Theorems \ref{lemmHComplete} and \ref{lem:relatively complete} yield:

\begin{theorem}\label{thm:baire is monadic}
If $X$ is a topological space, then $ \baire X := (\bquot X,\C)$ is a monadic algebra, where for $a\in\bquot X$, $\C a$ is defined as the least closed element above $a$.
\end{theorem}

\begin{remark}
In general, for a closure algebra $(B,\C)$, an element $a\in B$ is {\em closed} if $a = \C a$, {\em open} if $a = \I a$, and {\em clopen} if it is both closed and open.
Since $\baire X$ is a closure algebra, it is not hard to check that these notions coincide with the ones given in Definition~\ref{defClopen}. 
\end{remark}

\section{Baire resolvability}




To obtain completeness results with respect to this new semantics we require to introduce the concept of Baire resolvability, which refines the concept of resolvability introduced by Hewitt \cite{Hew43}.
A space is {\em resolvable} if it can be partitioned into disjoint dense sets. 
This notion readily extends to the algebraic setting:

\begin{definition}
Let $B$ be a Boolean algebra. We call $a,b\in B$ \emph{orthogonal} if $a\wedge b=0$, and a family $\{a_i\mid i\in I\}$ \emph{pairwise orthogonal} if $a_i\wedge a_j=0$ for $i\ne j$.
\end{definition}

\begin{definition}
Let $(B,\C)$ be a closure algebra and $\kappa$ a nonzero cardinal.
\begin{enumerate}
\item We call $a\in B$ \emph{dense} if $\C a=1$.
\item We call $(B,\C)$ \emph{$\kappa$-resolvable} if there is a pairwise orthogonal family of dense elements whose cardinality is $\kappa$.
\end{enumerate}
\end{definition}

\begin{definition}
Let $(B,\C)$ be a closure algebra, $a\in B$, and $\kappa$ a nonzero cardinal. We call $a$ {\em $\kappa$-resolvable} if there is a pairwise orthogonal family $\mathcal R = \{b_\iota\mid \iota<\kappa\}$ such that $b_\iota\leq a$ and $a\le\C b_\iota$ for each $\iota<\kappa$.
The family $\mathcal R$ is a {\em $\kappa$-resolution of $a$.} 
\end{definition}

If $\{[A_\iota] \mid \iota<\kappa\}$ is a $\kappa$-resolution of $\baire X$, we remark that the collection of representatives $A_\iota$ need not come from a partition of $X$; for example, two of them may have nonempty (but meager) intersection.
However, sometimes it will be useful to have `nice' representatives, as made precise in the following definition.

\begin{definition}\label{defBaireRes}
Let $X$ be a Baire space and let $Y\subseteq X$.
If $A\subseteq Y$ is such that $[Y]\sqsubseteq \C[A]$, we say that $A$ is {\em Baire-dense in $Y.$} 
For a nonzero cardinal $\kappa$, we call $Y  \subseteq   X$ \emph{Baire $\kappa$-resolvable} if there is a partition $\{A_\iota \mid \iota<\kappa\}$ of $Y$ such that, for each $\iota$, $ A_\iota $ is Baire-dense in $Y$.  
We call $\{A_\iota \mid \iota<\kappa\}$ a {\em Baire $\kappa$-resolution of $Y$.}
\end{definition}

If $\{A_\iota \mid \iota<\kappa\}$ is a Baire $\kappa$-resolution of $Y$, then $\{ [A_\iota] \mid \iota<\kappa\}$ is a $\kappa$-resolution of $[Y]$.
Thus, Baire $\kappa$-resolvability of $Y$ implies $\kappa$-resolvability of $[Y]$.
As we have mentioned, the converse is not true a priori since the $A_\iota$ do not have to form a partition, but it is not hard to check that the two notions are equivalent if $\kappa$ is countable: just define recursively $A'_0=A_0$ and $A'_{i+1} = A_{i+1}\setminus \bigcup_{j<i}A'_j$.  
Note that $a_\iota=[A_\iota]$ is dense in $\baire X$ provided for each nonempty open $U\in\wp(X)$ we have $A_\iota\cap U \notin\meager$. We call such sets `nowhere meager'.
More generally, we adopt the following conventions.

\begin{definition}\label{defNowhere}
Let $X$ be a topological space and $\kappa$ a cardinal. We say that: 
\begin{itemize}

\item $A\subseteq X$ is {\em nowhere meager} if for each nonempty open $U \subseteq X$ we have that $A \cap U \notin\meager$.

\item $A\subseteq X$ is {\em somewhere meager} if it is not nowhere meager.

\item $X$ is {\em everywhere Baire $\kappa$-resolvable} if each open $U\subseteq X$ is Baire $\kappa$-resolvable (seen as a subspace of $X$).

\item $X$ is {\em somewhere Baire $\kappa$-resolvable} if there is some nonempty open $U\subseteq X$ that is Baire $\kappa$-resolvable.

\item $X$ is {\em nowhere Baire $\kappa$-resolvable} if it is not somewhere Baire $\kappa$-resolvable.

\end{itemize}
\end{definition}

So, $A\subseteq X$ is somewhere meager iff there is a nonempty open $U\subseteq X$ with $A\cap U\in\meager$. 
We have the following characterizations of Baire resolvability:

\begin{lemma}\label{lemmDefRes}
A Baire space $X$ is Baire $\kappa $-resolvable iff it can be partitioned into $\kappa$ many nowhere meager sets.
\end{lemma}

\begin{lemma}\label{lemmGlobalToLocalRes}
Let $X$ be a topological space and $\kappa$ a nonzero cardinal. Then $X$ is Baire $\kappa$-resolvable iff it is everywhere Baire $\kappa$-resolvable.
\end{lemma}

\proof
Clearly if $X$ is everywhere Baire $\kappa$-resolvable, then $X$ is Baire $\kappa$-resolvable. Conversely, 
suppose that $X$ is Baire $\kappa$-resolvable and $\{A_\iota\mid \iota<\kappa\}$ is a Baire $\kappa$-resolution of $X$. Let $U \subseteq X$ be nonempty open. 
Since $A_\iota$ is nowhere meager, $A_\iota\cap U$ is nowhere meager in $U$. Thus, $\{A_\iota\cap U\mid \iota<\kappa\}$ is a Baire $\kappa$-resolution of $U$.
\endproof

Let $X$ be a Baire space and $\kappa\le\omega$.
If $X$ is Baire $\kappa$-resolvable, then $X$ is $\kappa$-resolvable, given that Baire-dense sets must be dense. 
The next example shows that the converse is not true in general.

\begin{example}\label{exOmPluOne}
Let $X = \omega+1$ with the topology whose nonempty open sets are of the form $[n,\omega]$ where $n<\omega$. Then $\{\omega\}$ is dense in $X$, and being a singleton, it cannot be written as a countable union of nowhere dense sets, so $\{\omega\}\notin\meager$. Therefore, every nonempty open set in $X$ is not meager, and hence $X$ is a Baire space.

Each infinite subset of $[0,\omega)$ is dense. Therefore, if we write $[0,\omega)$ as a countable union of infinite disjoint subsets of $[0,\omega)$, we obtain that $X$ is $\omega$-resolvable. On the other hand, $[0,\omega)=\bigcup_{n<\omega}[0,n]$ and each $[0,n]$ is nowhere dense, so $[0,\omega)\in\meager$. Thus, no subset of $[0,\omega)$ is nowhere meager, yielding that $X$ is not Baire resolvable.
\end{example}

As we pointed out in the Introduction, our main Baire resolvability results follow from the
Disjoint Refinement Lemma (see \cite[Lem.~7.5]{CN74} and \cite[Notes for \S 7]{CN74} for the history of this result).
This lemma has been used in proofs of resolvability in other contexts (see e.g.~\cite{Comfort96}); ours follows a similar pattern.

\begin{lemma}[Disjoint Refinement Lemma]\label{lemmDRL}
Let $\kappa$ be an infinite cardinal and $\{ A_\iota  \mid \iota<\kappa \}$ a family of sets such that $|A_\iota|=\kappa$ for each $\iota<\kappa$. Then there is a family $\{ B_\iota \mid \iota<\kappa \}$ of pairwise disjoint sets such that $B_\iota\subseteq A_\iota$ and $|B_\iota|=\kappa$ for each $\iota<\kappa$.
\end{lemma}

We will use the following version of the Disjoint Refinement Lemma:

\begin{lemma}\label{lemmDisjRef}
Let $\kappa$ be an infinite cardinal, $X$ a set, and $\mathcal K$ a family of subsets of $X$ such that $|\mathcal K| \leq \kappa$ and $|K| \geq \kappa$ for each $K\in\mathcal K$. Then there is a partition $\{ A_\iota \mid \iota < \kappa \}$ of $X$ such that $ A_\iota \cap K  \neq \varnothing$ for each $K\in\mathcal K$ and $\iota < \kappa$. 
\end{lemma}

\proof
Enumerate $\mathcal K$ as $\{ K_\iota \mid \iota <\kappa \}$, and for each $\iota<\kappa$ let $K'_\iota$ be a subset of $K_\iota$ with $|K'_\iota| = \kappa$.
By the Disjoint Refinement Lemma, there is a family $\{ B_\iota \mid \iota<\kappa \}$ of pairwise disjoint sets such that $B_\iota\subseteq K'_\iota$ and $|B_\iota|=\kappa$ for each $\iota<\kappa$.
Enumerating each $B_\iota$ as $\{b_{\iota\nu}\mid \nu<\kappa  \}$, for $\nu<\kappa$, let $A'_\nu = \{b_{\iota\nu}\mid \iota<\kappa \}$. 
To ensure that we obtain a partition, let $A_0 = A'_0\cup \left( X \setminus \bigcup_{\nu <\kappa} A'_\nu \right)$ and for $\nu >0$, let $A_\nu  = A'_\nu $. 
Then $\{ A_\nu \mid \nu <\kappa \}$ is the desired partition, since for each $\iota<\kappa$ we have that $b_{\iota\nu} \in A_\nu \cap K'_\iota \subseteq A_\nu \cap K_\iota$.
\endproof

To apply Lemma \ref{lemmDisjRef} to Baire resolvability, we use the notion of a {\em $\kappa$-witnessing family.}

\begin{definition}
Let $X$ be a Baire space and $\kappa$ a cardinal. We call $\mathcal K \subseteq \wp(X)$ a {\em $\kappa$-witnessing family} for $X$ provided $|\mathcal K|\leq \kappa$, $|K| \geq \kappa$ for each $K\in\mathcal K$, and if $A\subseteq X$ is somewhere meager, then there is $K\in\mathcal K$ with $K\cap A = \varnothing$.
\end{definition}

\begin{theorem}\label{theoBaireResolve}
Let $\kappa$ be an infinite cardinal.
Each Baire space that has a $\kappa$-witnessing family is Baire $\kappa$-resolvable.
\end{theorem}

\proof
Let $X$ be a Baire space with a $\kappa$-witnessing family $\mathcal K$. By Lemma~\ref{lemmDisjRef}, there is a family $\{ A_\iota \mid \iota<\kappa \}$ of pairwise disjoint subsets of $X$ such that $ A_\iota \cap K \neq \varnothing$ for each $K\in\mathcal K$ and $\iota<\kappa$. We claim that each $A_\iota$ is nowhere meager. Otherwise, since $\mathcal K$ is a $\kappa$-witnessing family, there is $K\in\mathcal K$ such that $K\cap A_\iota=\varnothing$, a contradiction. Thus, $\{ A_\iota \mid \iota<\kappa \}$ gives the desired Baire $\kappa$-resolution of $X$.
\endproof

In the next two lemmas we show that the uncountable compact sets provide a witnessing family in many Baire spaces.
These lemmas are folklore, but we state them in the precise form we need and briefly recall their proofs.

\begin{lemma}\label{lemmCompUnc}
Let $\kappa = \power{\aleph_0}$ and $X$ be a crowded topological space which is either locally compact Hausdorff or completely metrizable. If $A\subseteq X$ is somewhere meager, then there is a compact $K\subseteq X$ such that $|K|\geq \mathfrak c$ and $A\cap K=\varnothing$.
\end{lemma}

\proof
The proof follows the standard tree construction method (see, e.g., \cite[Exercise 3.12.11]{Eng89}).
We provide a sketch in the case when $X$ is a crowded locally compact Hausdorff space.
Suppose that $A\subseteq X$ is somewhere meager, so there is a nonempty open $U\subseteq X$ such that $A\cap U\in\meager$.
Then $A\cap U=\bigcup_{n<\omega} B_n$, with each $B_n$ nowhere dense. Using the assumption that $X$ is crowded locally compact Hausdorff, to each finite sequence ${\bf b} = (b_0,\hdots,b_n) \in \{ 0,1 \}^{<\omega}$, we can assign a compact set $K_{\bf b} \subseteq U$ so that
\begin{itemize}
\item $K_{\bf b}$ has nonempty interior;

\item $B_i \cap K_{\bf b} = \varnothing$ for each $i\leq n$;

\item if ${\bf b}$ is an initial segment of ${\bf b}'$, then $K_{{\bf b}'} \subseteq K_{\bf b}$;

\item if $\bf b$, $\bf b'$ are incomparable (i.e.~disagree on some coordinate), then $K_{\bf b} \cap K_{{\bf b}'} = \varnothing$.
\end{itemize}

Let
$K = \bigcap_{n <\omega} \bigcup_{|{\bf b}| = n} K_{\bf b}.$
Being an intersection of compact sets, this is a compact set.
It follows from the construction that each $K_{\bf b}$ is contained in $U$ and that $A\cap K=\varnothing$.
The intersections along infinite sequences ${\bf b} \in \{0,1\}^{\omega}$ are nonempty and disjoint, witnessing that $|K| \geq \mathfrak c$.

The proof for a crowded completely metrizable $X$ is essentially the same, but we let each $K_{\bf b}$ be a ball of radius at most $2^{-|{\bf b}|}$ so that we can use that $X$ is complete. 
\endproof

\begin{lemma}\label{lemmCompactCount}
Let $X$ be either Hausdorff and second-countable or metrizable and of cardinality continuum. Then $X$ contains at most $\mathfrak c$ compact sets.
\end{lemma}

\proof
First suppose that $X$ is Hausdorff and second-countable. The latter implies that $X$ contains at most $\mathfrak c$ open sets, and so
at most $\mathfrak c$ closed sets. Since $X$ is Hausdorff, compact sets are closed, and hence there can be at most $\mathfrak c$ compact sets.

Next suppose that $X$ is metrizable and of cardinality continuum.
Then every compact set $K\subseteq X$ is 
the closure of some countable subset of $X$ (see, e.g., \cite[Thm.~4.1.18]{Eng89}). 
Since $X$ has $\mathfrak c$-many countable subsets, there can be at most $\mathfrak c$ values for $K$.
\endproof

Putting together these results, we arrive at the following:

\begin{theorem}\label{theoResolve}
Suppose that $X$ is a crowded space which is either locally compact, Hausdorff, and second-countable or completely metrizable and continuum sized. Then $X$ is Baire $\mathfrak c$-resolvable.
\end{theorem}

\proof
By Lemmas~\ref{lemmCompUnc} and~\ref{lemmCompactCount}, the uncountable compact subsets of $X$ form a $\mathfrak c$-witnessing family for $X$. Therefore, by Theorem~\ref{theoBaireResolve}, $X$ is Baire $\mathfrak c$-resolvable.
\endproof

\begin{remark}
Lemma \ref{lemmCompUnc} can be generalized to the setting of \v Cech-complete spaces, under the assumption that no point has a least neighborhood (equivalently, that every point has infinite character: see, e.g., \cite[Exercise 3.12.11(b)]{Eng89}).
This allows us to extend Theorem \ref{theoResolve} to include any second-countable \v Cech-complete space with the latter property.
\end{remark}

\section{New completeness results for $\sf S5$}

In this section we derive new completeness results for $\sf S5$ and its extensions using our new semantics of Baire algebras.
We start by recalling that a map $h \colon A\to B$ between two closure algebras $(A,\C_A)$ and $(B,\C_B)$ is a {\em homomorphism of closure algebras} if $h$ is a homomorphism of Boolean algebras and $h (\C_A a) = \C_B h(a)$ for each $a\in A$. We say that $h$ is an {\em embedding} if it is injective, and that $h$ is an {\em isomorphism} if it is bijective.

Recall that if $X$ and $Y$ are topological spaces and $f\colon X\to Y$, then $f$ is {\em continuous} if $f^{-1}(B)$ is open whenever $B\subseteq Y $ is open, {\em open} if $f(A)$ is open whenever $A\subseteq X$ is open, and an {\em interior map} if it is both continuous and open.
It is well known (see, e.g., \cite[p.~99]{RS63}) that if $f:X\to Y$ is an interior map, then $f^{-1}:\Kur(Y)\to\Kur(X)$ is a homomorphism of closure algebras. Moreover, if $f$ is surjective, then $f^{-1}$ is an embedding. To obtain analogous results for Baire algebras, it is sufficient to work with partial maps.
We remind the reader that $[A]$ denotes the equivalence class of $A$ modulo the meager sets.

\begin{definition}\label{defBaireMap}
Let $X,Y$ be arbitrary topological spaces and $f:X\to Y$ a partial map.
\begin{enumerate}
\item We say that $f$ is defined {\em almost everywhere} if the complement of the domain of $f$ is meager.
\item We call $f$ {\em non-degenerate} if $B\subseteq Y$ meager implies that $f^{-1}(B)$ is meager. 
\item We define $h\colon \baire Y \to \baire X$ by
\[
h [ A ] = [f^{-1} (A)] \mbox{ for each } A\subseteq Y.
\]
\end{enumerate}
\end{definition}

\begin{lemma}\label{lem:h well defined}
Let $X ,Y$ be arbitrary topological spaces and $f\colon X \to Y$ a partial map. If $f$ is defined almost everywhere and is non-degenerate,
then $h\colon \baire Y \to \baire X$ is a well-defined homomorphism of Boolean algebras.
\end{lemma}

\proof
Let $A,B \subseteq Y$ with $[A] = [B]$. Then $A\setminus B$ is meager, so $f^{-1}(A) \setminus f^{-1}(B) = f^{-1} (A\setminus B)$ is meager since $f$ is non-degenerate. By a symmetric argument, $f^{-1}(B) \setminus f^{-1}(A)$ is also meager.
Thus, $[f^{-1}(A)] = [f^{-1}(B)]$, and hence $h$ is well defined.

In addition,
\[
h([A] \sqcap [B]) = [f^{-1}(A \cap B) ] =[f^{-1}(A)] \sqcap [f^{-1}(B) ] = h(A)\sqcap h(B).
\]
Therefore, $h$ commutes with binary meets. To see that it also commutes with complements, let $M$ be the complement of the domain of $f$. Then $X \setminus f^{-1}(A) = f^{-1}(Y \setminus A)\cup M$. Since $f$ is defined almost everywhere, $M$ is meager, so $[ X \setminus f^{-1}(A) ] = [ f^{-1}(Y \setminus A) ]$. Thus,
\[
- h [A] = [ X \setminus f^{-1}(A) ] = [ f^{-1}(Y \setminus A) ] = h (- [A]).
\]
Consequently, $h$ is a well-defined homomorphism of Boolean algebras.
\endproof

Note, however, that $h$ may fail to be injective.
To obtain an embedding of Boolean algebras, we need an additional condition on $f$.

\begin{definition}
We call a partial map $f\colon X\to Y$ {\em exact} if $f^{-1}(A)$ meager implies that $A$ is meager for each $A\subseteq Y$.
\end{definition}

\begin{lemma}\label{lem:embed}
Let $X,Y$ be Baire spaces and let $f\colon X\to Y$ be non-degenerate, exact, and defined almost everywhere. Then $h\colon \baire Y \to \baire X$ is an embedding of Boolean algebras.
\end{lemma}

\proof
By Lemma~\ref{lem:h well defined}, $h$ is a homomorphism of Boolean algebras. To see that it is an embedding, let $h [ A ] = h [ B ]$.
Then $ f^{-1}(A) \setminus f^{-1}(B) = f^{-1} (A\setminus B)$ is meager.
Since $f$ is exact, $A \setminus B $ is meager. A symmetric argument yields that $B\setminus A $ is meager. Thus, $[A] = [B]$, and hence $h$ is an embedding.
\endproof

Next we detail some conditions that will ensure that $h$ is not only a Boolean homomorphism, but indeed a homomorphism of closure algebras.

\begin{definition}\label{DefBaireMap}
Let $f\colon X \to Y$ be a partial map.
We say that
\begin{enumerate}
\item $f$ is {\em Baire-continuous} if $V$ an open subset of $Y$ implies that $[ f^{-1}(V) ]$ is an open element of $\baire X$;

\item $f$ is {\em Baire-open} if $U$ a nonempty open subset and $M$ a meager subset of $X$ imply that $[f( U \setminus M )]$ is a nonzero open element of $\baire Y$;

\item $f$ is a {\em Baire map} if it is defined almost everywhere and is non-degenerate, Baire-continuous, and Baire-open.
\end{enumerate}
\end{definition}

\begin{remark}
If a Baire map $f\colon X\to Y$ exists, then $X$ must be a Baire space.
This is because if $U\subseteq X$ were a nonempty open meager set, then $ [ f(U\setminus U) ] = 0$ in $\baire Y$, so $f$ could not be Baire-open. 
\end{remark}

\begin{lemma}\label{lem:BaireIsEmbed}
Let $X,Y$ be topological spaces. If $f\colon X\to Y$ is a Baire map, then $h\colon \baire Y \to \baire X$ is a homomorphism of closure algebras. If in addition $f$ is exact, then $h$ is an embedding of closure algebras.
\end{lemma}

\proof
By Lemma~\ref{lem:h well defined}, $h$ is a 
homomorphism of Boolean algebras. Let $A\subseteq Y$. We first show that $\C h [ A ] \sqsubseteq h   \C [ A ]   $.
Since $\C [ A]$ is closed, there is a closed set $D\subseteq Y$ such that $[D] = \C [ A]$. Because $h$ is a homomorphism of Boolean algebras, it is order preserving. Therefore, $[A] \sqsubseteq \C [ A ]$ implies $h [ A ] \sqsubseteq h   \C[ A ]  = [f^{-1}( D )]$. Let $U=Y\setminus D$ and recall that $M$ is the complement of the domain of $f$. Since $U$ is open and $M$ is meager in $Y$, we have
\[
[ f^{-1}(D) ] = [ f^{-1}(Y\setminus U) ] = [ f^{-1}(Y\setminus U) \cup M ] = [ X\setminus f^{-1}(U) ] = - [f^{-1}(U)].
\]
Because $f$ is Baire-continuous and $U$ is open, $[f^{-1}(U)]$ is open, hence $[f^{-1}( D )]$ is closed in $\baire X$. Thus, $\C h [ A ] \sqsubseteq [f^{-1}( D )] = h \C [A] $ by the definition of closure.

We next show that $h \C [A]  \sqsubseteq \C h [ A ]$. There is a closed set $C \subseteq X$ such that $\C h [A] = [C]$.
Let $D$ be as above and, using the fact that $f$ is Baire-continuous, write $[f^{-1}(D)] = [F]$, where $F\subseteq X$ is closed.
Let $V = \I F \setminus C  $. Since $ F $ is closed in $X$,
we have $[ \I F ] = [F] = [ f^{-1} (D) ]$. Therefore,
\[
[V] = [ f^{-1} (D) ] - [ C ] =
h\C [A] -  \C h[ A],
\]
so it suffices to show that $V=\varnothing$. Suppose otherwise. Since $[f^{-1}(A)] \sqsubseteq [C]$, we have that $f^{-1}(A)\setminus C$ is meager, so $L:= f^ {-1} (A) \cap V$ is meager. Because $f$ is Baire-open, $L$ is meager, and $V$ is nonempty open, 
$[f(V) \setminus A] = [f( V\setminus f^{-1} (A))] =[f( V\setminus L)]$ is nonzero and open in $\baire Y$.
Thus, $\C [A] - [f(V) \setminus A] \sqsubset \C [A]$.
Recall that we chose $D$ so that $\C [A] = [D]$. In particular, $[A] \sqsubseteq [D]$.
Since $A \setminus (D \setminus (f(V) \setminus A)) = A\setminus D$ and the latter is meager, we see that
\[[A] \sqsubseteq [ D \setminus (f(V) \setminus A) ] = \C [A] - [f(V) \setminus A] \sqsubset \C [ A].\]
As $\C [A] - [f(V) \setminus A]$ is closed in $\baire Y$, this contradicts the definition of $\C[A]$. Consequently, $V = \varnothing$, as required.

Finally, if $f$ is exact, then it follows from Lemma~\ref{lem:embed} that $h$ is an embedding.
\endproof

If $Y$ is an $\sf S5$-frame, it follows from Example \ref{examS4}(1) that $\Kur(Y)$ is isomorphic to $\baire Y$, so we may regard $h$ as a map from $\Kur(Y)$ to $\baire X$.
Since the only meager subset of an $\sf S5$-frame is $\varnothing$, the next lemma is immediate.

\begin{lemma}\label{lemmBaireIsEmbed}
Let $X$ be a Baire space, $Y$ an $\sf S5$-frame, and $f\colon X\to Y$ a Baire map.
Then $f$ is exact iff $f^{-1}( y )$ is non-meager for each $y\in Y$.
\end{lemma}

Recall that for a nonzero cardinal $\kappa$, $\mathfrak C_\kappa$ denotes the $\kappa$-cluster. By \cite[Lem.~5.9]{BBBM17}, a space $X$ is $\kappa$-resolvable iff there is an interior map from $X$ onto $\mathfrak C_\kappa$.
We have the following analogue of this result for Baire algebras.

\begin{lemma}\label{corEmbed}
Let $X$ be a Baire space and $\kappa$ a nonzero cardinal. Then $X$ is Baire $\kappa$-resolvable iff there is an exact Baire map $f\colon X \to \mathfrak C_\kappa$.
\end{lemma}

\proof
Let $(w_\iota)_{\iota<\kappa}$ be an enumeration of $\mathfrak C_\kappa$.
First suppose that $X$ is Baire $\kappa$-resolvable, and let $(A_\iota)_{\iota<\kappa}$ be a Baire $\kappa$-resolution of $X$. Define $f\colon X\to \mathfrak C_\kappa$ by setting $f(x) = w_\iota$ provided $x\in A_\iota$. Then $f$ is a total onto map. Since $f$ is total, it is defined almost everywhere, and $f$ is non-degenerate
because $\varnothing$ is the only meager subset of $\mathfrak C_\kappa$.
Since $\varnothing,\mathfrak C_\kappa$ are the only opens of $\mathfrak C_\kappa$, it is clear that $f$ is continuous, hence Baire-continuous.
To see that $f$ is Baire-open, let $U\subseteq X$ be nonempty open and $M\subseteq X$ meager.
Since each $A_\iota$ is Baire-dense, $A_\iota\cap (U\setminus M)$ is non-meager, 
hence nonempty, so $w_\iota \in f(U \setminus M)$, and hence $f(U \setminus M) = \mathfrak C_\kappa$.
Therefore, $f$ is a Baire map. Finally, $f^{-1}(w_\iota) = A_\iota$, which is non-meager. Thus, $f$ is exact by Lemma~\ref{lemmBaireIsEmbed}.

Next suppose that $f\colon X\to \mathfrak C_\kappa$ is an exact Baire map.
By Lemma~\ref{lemmBaireIsEmbed}, each $f^{-1}(w_\iota)$ is non-meager, and we can enlarge $f^{-1}(w_0)$ to obtain a partition of $X$ if needed. By identifying $\baire {\mathfrak C_\kappa}$ with ${\bf Kur}(\mathfrak C_\kappa)$ and applying Lemma~\ref{lem:BaireIsEmbed}, $h:{\bf Kur}(\mathfrak C_\kappa) \to \baire X$ is a homomorphism of closure algebras. Therefore, for each $\iota<\kappa$,
\[
\C [f^{-1}(w_\iota)] = \C h(w_\iota) = h \C (w_\iota) = h(\mathfrak C_\kappa) = [f^{-1}(\mathfrak C_\kappa)] = [X].
\]
Thus, $( f^{-1}(w_\iota) )_{\iota<\kappa}$ witnesses that $X$ is Baire $\kappa$-resolvable.
\endproof

We are ready to prove our first completeness result. 
For this we recall that ${\sf S5}^\lambda $ denotes the variant of $\sf S5$ with $\lambda$-many variables, and if $\lambda<\omega$, then ${\sf S5}_\lambda$ is the extension of $\sf S5$ axiomatized by $\sfiven \lambda$.

\begin{theorem}\label{theoComplete}
Let $X$ be a Baire space, $\lambda$ a nonzero cardinal, and suppose that $X$ is Baire $\lambda$-resolvable. 

\begin{enumerate}[label=\emph{ (\arabic*)}]
\item\label{itCompleteTwo} If $\lambda \geq \omega$, then ${\sf S5}^\lambda $ is sound and strongly complete for $\baire X$.

\item\label{itCompleteOne} If $\lambda < \omega$ and $X$ is nowhere Baire $(\lambda+1)$-resolvable, then $ {\sf S5}_\lambda$ is sound and strongly complete for $\baire X$.
\end{enumerate}
\end{theorem}
\color{black}

\proof
\noindent (1)
Soundness follows from the fact that $\baire X$ is a monadic algebra (see Theorem~\ref{thm:baire is monadic}).
For completeness, let $\Gamma \models_{\baire X} \varphi$. By Lemma~\ref{corEmbed}, there is an exact Baire map $f\colon X\to \mathfrak  C_\lambda$. Since ${\bf Kur}( \mathfrak C_\lambda)$ is isomorphic to $\baire { \mathfrak C_\lambda}$, by Lemma~\ref{lem:BaireIsEmbed}, ${\bf Kur}( \mathfrak C_\lambda)$ embeds into $\baire X$. Therefore,
$\Gamma \models_{ \mathfrak C_\lambda} \varphi$. Thus, 
$\Gamma \vdash_{{\sf S5}^\lambda} \varphi$ by Theorem~\ref{theoS5n}.

\smallskip

\noindent(2)
We first prove completeness. Suppose $\Gamma \models_{\baire X} \varphi$. By Lemma~\ref{corEmbed}, there is an exact Baire map $f\colon X\to  \mathfrak C_\lambda$. By Lemma~\ref{lem:BaireIsEmbed}, ${\bf Kur}( \mathfrak C_\lambda)$ embeds into $\baire X$. Therefore, $\Gamma \models_{ \mathfrak C_\lambda} \varphi$.
Thus, 
$\Gamma \vdash_{{\sf S5}_\lambda} \varphi$ by Theorem~\ref{theoS5n}.

To prove soundness,
suppose that $\sf S5_\lambda$ is not sound for $X$, and fix a valuation $\val\cdot$ falsifying the $\sf S5_\lambda$ axiom $\sfiven \lambda$.
Since this axiom is a Boolean combination of formulas $\Diamond \psi$, we have that $\val{ \sfiven \lambda }$ is clopen, so we can choose nonempty open $U$ with $[U] = \val{ \neg \sfiven \lambda }$.
Then
\begin{equation}\tag{\textasteriskcentered}\label{eqAnte}
[U] \sqsubseteq \val{\bigwedge_{i = 0}^\lambda \D p_i}
\end{equation}
and
\begin{equation}\tag{\dag}\label{eqConse}
[U] \sqcap \val{ \bigvee_{0\leq i<j\leq \lambda} \D(p_i\wedge p_j) } = 0.
\end{equation}
For $i\leq \lambda$, let $P'_i$ be such that $\val {p_i} = [P'_i]$.
From \eqref{eqAnte} we obtain for each $i\le\lambda$ that $[U] \sqsubseteq \C [P'_i ]$, which means that $P'_i$ is nowhere meager in $U$, i.e.~if $U'\subseteq U$ is nonempty open then $P'_i\cap U'$ is non-meager. 
From \eqref{eqConse} it follows that $[U]\sqcap \C  \val {p_i\wedge p_j}    = 0$  whenever $i<j\leq \lambda$, so $U\cap P'_i \cap P'_j$ is meager. Since $\lambda$ is finite, we can then replace each $P'_i$ by some $P_i$ in such a way that $[P_i] = [P'_i]$ and $\{P_i \mid i \leq \lambda \}$ forms a partition of $U$, witnessing that $U$ is Baire $(\lambda + 1)$-resolvable, hence $X$ is somewhere Baire $(\lambda + 1)$-resolvable, contrary to our assumption. 
\endproof

As a consequence of Theorems~\ref{theoComplete}(1) and~\ref{theoResolve}, we arrive at the following analogue of the McKinsey-Tarski theorem:

\begin{corollary}\label{corComplete}
Let $X$ be a crowded space which is either completely metrizable and of cardinality continuum, or second-countable locally compact Hausdorff. Then ${\sf S5}^{\mathfrak c}$ is sound and strongly complete for $\baire X$.
\end{corollary}

In particular, we obtain:

\begin{corollary}
$\sf S5$ is the logic of $\baire X$ whenever $X$ is:
\begin{enumerate}
\item[{\em (1)}] $\mathbb R^n$ for any $n \geq 1$.

\item[{\em (2)}] Any perfect closed subset of $\mathbb R^n$, including the Cantor space and the interval $[0,1]$.

\item[{\em (3)}] The Banach space $\ell^p$ for $p\in [1,\infty]$.
\end{enumerate}
\end{corollary}

Recall that $\ell^p$ is the normed vector space whose elements are functions $f\colon\mathbb N\to\mathbb R$ whose $p$-norm converges (see~\cite{CohnMeasure} for details).
Note that we may allow $p=\infty$ since $|\ell^\infty| = \mathfrak c$, despite not being second-countable.

\begin{remark}\label{remQ}
On the other hand, since $\mathbb Q$ is meager, $\baire {\mathbb Q}$ is the trivial algebra, and hence its logic is the contradictory logic, i.e.~the logic axiomatized by $\bot$.
This is in contrast with the McKinsey-Tarski theorem that the logic of ${\bf Kur}(\mathbb Q)$ is $\sf S4$.
\end{remark}

By Theorem \ref{theoComplete}(2), if $X$ is a Baire space that is Baire $n$-resolvable and nowhere Baire $(n+1)$-resolvable, then the logic of $\baire X$ is ${\sf S5}_n$. As we next show, we can avoid verifying whether $X$ is nowhere Baire $(n+1)$-resolvable by constructing a subalgebra of $\baire X$ whose logic is ${\sf S5}_n$. 
For this we require the following decomposition lemma, which follows from similar decompositions in the theory of $\ell$-groups; see, e.g., \cite{BMO20} and the references therein. To keep the paper self-contained, we give a proof.

\begin{lemma}\label{lem:generation}
Let $A$ be a monadic algebra, $A_0$ its subalgebra of all clopen elements, $B$ a Boolean subalgebra of $A$, and $C$ the Boolean subalgebra of $A$ generated by $A_0\cup B$.
\begin{enumerate}[label=\emph{ (\arabic*)}]
\item\label{itgenerationone} $C$ is a monadic subalgebra of $A$.
\item\label{itgenerationtwo} Each $c\in C$ can be written as $c=\bigvee_{i=1}^n(a_i\wedge b_i)$, where $a_1,\dots,a_n\in A_0$ are pairwise orthogonal and $b_1,\dots,b_n\in B$.
\item\label{itgenerationthree} Each $c^1,\ldots,c^n\in C$ can be written in the compatible form $c^j=\bigvee_{i=1}^n(a_i\wedge b^j_i)$ where $a_1,\dots,a_n\in A_0$ are pairwise orthogonal and each $b^j_i\in B$.
\end{enumerate}
\end{lemma}

\proof
\noindent (1) For $a\in C$ we have $\C a\in A_0 \subseteq C$ (since closed elements are clopen by Theorem~\ref{lemmHComplete}). Thus, $C$ is a monadic subalgebra of $A$.

\medskip

\noindent (2) Since $C$ is a Boolean subalgebra of $A$, it is well known (see, e.g., \cite[p.~74]{RS63}) that each $c\in C$ can be written as $c=\bigvee_{i=1}^n\bigwedge_{j=1}^{m_j}d_{ij}$ where either $d_{ij}\in A_0\cup B$ or $-d_{ij}\in A_0\cup B$. Since both $A_0$ and $B$ are Boolean subalgebras, we may assume that $d_{ij}\in A_0\cup B$; and gathering together the elements in $A_0$ and $B$, we may write $\bigwedge_{j=1}^{m_j}d_{ij}=a_i\wedge b_i$ where $a_i\in A_0$ and $b_i\in B$. Therefore, each $c\in C$ can be written as $c=\bigvee_{i=1}^n(a_i\wedge b_i)$ where $a_i\in A_0$ and $b_i\in B$. It is left to prove that the $a_i$ can be chosen pairwise orthogonal. But this is a standard argument. Indeed, if $c=(a_1\wedge b_1)\vee(a_2\wedge b_2)$, then we may write
\begin{eqnarray*}
c &=& (a_1\wedge b_1)\vee(a_2\wedge b_2) \\
&=& \left[\left( (a_1 - a_2) \wedge b_1\right) \vee \left((a_1\wedge a_2)\wedge b_1\right)\right] \vee \left[\left((a_2 \wedge a_1) \wedge b_2\right) \vee \left((a_2 - a_1) \wedge b_2\right)\right] \\
&=& \left( (a_1 - a_2) \wedge b_1\right) \vee \left((a_1\wedge a_2)\wedge(b_1\vee b_2)\right) \vee \left((a_2 - a_1) \wedge b_2\right),
\end{eqnarray*}
where $a_1 - a_2, a_1\wedge a_2, a_2 - a_1\in A_0$ are pairwise orthogonal and $b_1, b_1\vee b_2, b_2\in B$. Now a simple inductive argument finishes the proof.

\medskip

\noindent (3) We only consider the case of two elements $c,d$ as the general case follows by simple induction.
By (2),
we may write $c=\bigvee_{i=1}^n(a_i\wedge b_i)$ where $a_1,\dots,a_n\in A_0$ are pairwise orthogonal and $b_1,\dots,b_n\in B$. Similarly $d=\bigvee_{j=1}^m(e_j\wedge f_j)$ where $e_1,\dots,e_m\in A_0$ are pairwise orthogonal and $f_1,\dots,f_m\in B$. By letting $a_{n+1}=\lnot\bigvee_{i=1}^n a_i$ and $b_{n+1}=0$ we may assume that $\bigvee_{i=1}^n a_i=1$, and similarly $\bigvee_{j=1}^m e_j=1$. But then
\[
c=\bigvee_{i=1}^n\bigvee_{j=1}^m((a_i\wedge e_j)\wedge b_i) \mbox{ and }
d=\bigvee_{j=1}^m\bigvee_{i=1}^n((a_i\wedge e_j)\wedge f_j).
\]
Clearly
\[
a_1\wedge e_1,\dots,a_1\wedge e_m,\dots,a_n\wedge e_1,\dots,a_n\wedge e_m
\]
are pairwise orthogonal and are in $A_0$; also each of $b_i,f_j$ are in $B$ (and may appear multiple times in the decomposition).
\endproof

\begin{theorem}\label{thmClosedS5n}
Let $X$ be a
Baire space and $n<\omega$. If $X$ is Baire $n$-resolvable, then $\baire X$ has a subalgebra $A$ containing all clopens of $\baire X$ such that $  {\sf S5}_n$ is sound and strongly complete for $A$.
\end{theorem}

\proof
Let $\mathfrak C_n=\{w_1,\ldots,w_n\}$ be the $n$-cluster.
By Lemmas~\ref{lem:BaireIsEmbed} and~\ref{corEmbed},
there is an embedding $h\colon {\bf Kur}(\mathfrak C_n)\to \baire X$. Let $A = h[{\bf Kur}(\mathfrak C_n)]$ and $H = \{h(w) \mid w\in \mathfrak C_n\}$. 
Then $A$ is a monadic subalgebra of $\baire X$ generated by $H$, and each $h(w)$ is an atom of $A$.
Let $B$ be the Boolean subalgebra of $\baire X$ generated by $A$ and the clopen elements of $\baire X$. By Lemma~\ref{lem:generation}(1), $B$ is a monadic subalgebra of $\baire X$. We claim that the logic of $B$ is ${\sf S5}_n$.

Since ${\sf S5}_n$ is complete for $\mathfrak C_n$, it follows that ${\sf S5}_n$ is complete for $B$.
It remains to check that ${\sf S5}_n$ is also sound for $B$.
For this it is sufficient to show that $\sfiven n$ is valid on $B$.
Let $\val\cdot$ be a valuation on $B$.
By Lemma~\ref{lem:generation}(3),
we may write each $\val{p_i} $ in the form $\bigsqcup_{ \ell = 1}^m (u_\ell  \sqcap a^i_\ell )$, where the $u_\ell $ are orthogonal clopens and the $a^i_\ell $ are elements of $A$. Therefore, 
since each $u_\ell$ is clopen, $\C(u_\ell \sqcap a_\ell^i) = u_\ell \sqcap \C a_\ell^i$, so we have
\begin{align*}
\val{\bigwedge _{i=1}^{n+1} \D p_i} = \bigsqcap_{i=1}^{n+1} \bigsqcup_{\ell = 1}^{m} \C (u_\ell  \sqcap a^i_\ell ) = \bigsqcap_{i=1}^{n+1} \bigsqcup_{\ell = 1}^{m} (u_\ell  \sqcap \C a^i_\ell )
=\bigsqcup_{\ell = 1}^{m} (u_\ell \sqcap \bigsqcap_{i=1}^{n+1} \C a^i_\ell ),
\end{align*}
where the last equality uses distributivity plus the orthogonality of the $u_\ell$, so that all cross-terms cancel.
Similarly,
\[
\val{\bigvee_{i\not = j} \D (p_i \wedge p_j)} = \bigsqcup_{i \neq j} \bigsqcup_{\ell = 1}^{m} (u_\ell  \sqcap  \C ( a^i_\ell  \sqcap a^j_\ell ) ) = \bigsqcup_{\ell = 1}^{m} (u_\ell  \sqcap \bigsqcup_{i \neq j} \C ( a^i_\ell  \sqcap a^j_\ell ) ).
\]
So to show that $\val\varphi = 1$, it suffices to show that
\begin{equation}\label{eqSqusbset}
\bigsqcup_{\ell = 1}^{m} (u_\ell \sqcap \bigsqcap_{i=1}^{n+1} \C a^i_\ell )
\sqsubseteq \bigsqcup_{\ell = 1}^{m} (u_\ell  \sqcap  \bigsqcup_{i \neq j} \C ( a^i_\ell  \sqcap a^j_\ell ) ) .
\end{equation}
Given that the $u_\ell$ are orthogonal, \eqref{eqSqusbset} holds iff, for each $\ell$,
\[
u_\ell \sqcap \bigsqcap_{i=1}^{n+1} \C a^i_\ell \sqsubseteq
u_\ell \sqcap \bigsqcup_{i \neq j} \C ( a^i_\ell \sqcap a^j_\ell ),
\]
and for this it suffices to show that
\begin{equation}\label{eqUL}
 \bigsqcap_{i=1}^{n+1} \C a^i_\ell \sqsubseteq
 \bigsqcup_{i \neq j} \C ( a^i_\ell \sqcap a^j_\ell ).
\end{equation}
So fix $\ell$.
Note that \eqref{eqUL} holds trivially if $a^i_\ell = 0$ for some $i$ (as the left-hand side is then zero), so we assume otherwise.
Since $A$ is generated by $H$, for each $i$ there is $w \in \mathfrak C_n$ such that $h(w ) \sqsubseteq a^i_\ell$.
Because $| \mathfrak C_n | = n$, the pigeonhole principle yields that there are $i\neq j$ and $w$ so that $h(w) \sqsubseteq a^i_\ell \sqcap a^j_\ell$.
But $\C h(w) = 1$, so $\C (a^i_\ell \sqcap a^j_\ell) = 1$.
Therefore,
\[
\bigsqcap_{i=1}^{n+1} \C a^i_\ell \sqsubseteq  1 = \C(a^i_\ell \sqcap a^j_\ell) \sqsubseteq
 \bigsqcup_{i \neq j} \C ( a^i_\ell \sqcap a^j_\ell ).
\]
Thus, \eqref{eqUL} holds and, since $\ell$ was arbitrary, we conclude that \eqref{eqSqusbset} holds, as needed.
\endproof

As a consequence of Theorems~\ref{thmClosedS5n} and~\ref{theoResolve} we obtain:

\begin{corollary}\label{corSubal}
Let $X$ be a crowded space which is either a complete metric space of cardinality $\mathfrak c$ or second-countable locally compact Hausdorff.
Then for each nonzero $n<\omega$ there is a subalgebra of $\baire X$ containing all clopen elements such that ${\sf S5}_n$ is sound and strongly complete for $\baire X$.
\end{corollary}

\section{New completeness results for $\sf S5U$}

In this final section, we extend our completeness results to the language with the universal modality (recall that the latter can be interpreted in arbitrary algebraic models as per Definition~\ref{DefSem}).  
As in the previous section, these results utilize standard Kripke completeness results.
In particular, the following completeness result is well known (see, e.g.,
\cite{GP92}), and can be lifted to strong completeness via standard methods.

\begin{theorem}\label{theoStrongCompCount}
$\sf S5U$ is sound and strongly complete for the class of countable $\sf S5$-frames.
\end{theorem}

In the topological setting, the language $\lang {}{\D\forall}$ is able to express connectedness by Shehtman's axiom
\[
\forall (\B p \vee \B \neg p)\rightarrow \forall p \vee \forall \neg p
\]
(see \cite{She99}).
Therefore, $\sf S4U$ is not complete for any connected space.
The situation is quite different in the setting of Baire algebras, and in order to see this, we will need to discuss disconnectedness in this context. This notion has already been considered in the context of closure algebras by McKinsey and Tarski \cite{MT44}.

\begin{definition}
A closure algebra $(B,\C)$ is {\em disconnected} if there are open $a,b \in B$ such that $a,b \not = 0$, $a \wedge  b = 0$, and $ a \vee  b = 1$.
\end{definition}

\begin{lemma}
If $X$ is a Hausdorff Baire space with at least two points, then $\baire X$ is disconnected.
\end{lemma}

\proof
Let $x\not = y$ in $X$, and let $U,V$ be the disjoint open neighborhoods of $x,y$.
Set $a = [U]$. Since $U$ is open, $a$ is open, hence clopen in $\baire X$. Because $X$ is a Baire space, $U$ is non-meager, so $a\ne 0$. Also, $X\setminus U$ is non-meager because $V\subseteq X\setminus U$. Thus, $a\ne 1$, and hence $\baire X$ is disconnected.
\endproof

As was the case with resolvability, disconnectedness readily extends to $\kappa$-dis\-connected\-ness for any cardinal $\kappa$.

\begin{definition}
A closure algebra $(B,\C)$ is {\em $\kappa$-disconnected} if there is a sequence $(a_\iota)_{\iota<\kappa} \subseteq B$ of nonzero pairwise orthogonal clopen elements such that $\bigvee_{\iota < \kappa } a_\iota = 1$.
\end{definition}

\begin{lemma}\label{lemmOmDisc}
If $X$ is an infinite Hausdorff Baire space, then $\baire X$ is $\omega$-disconnected.
\end{lemma}

\proof
First suppose that $X$ has no limit points.
Then every point is isolated, so each singleton $\{x_i\}$ is open, hence clopen and non-meager. 
Since $X$ is infinite, we can choose a sequence of distinct points $(x_i)_{i < \omega}$. 
Let $a_0 = [ X \setminus \{x_i\}_{1 \leq i<\omega} ] $ and for $i>0$ let $a_i = [ \{x_i\} ] $.
Then $(a_i)_{i<\omega}$ is a sequence of nonzero pairwise orthogonal clopens of $\baire X$,
witnessing that $X$ is $\omega$-disconnected.

Next suppose that $X$ has a limit point, say $x_\ast$. We define a sequence of points $(x_i)_{ i <\omega}$ and two sequences of open sets $(U_i)_{i \leq n}$ and $(V_i)_{ i \leq n}$ as follows.

To begin, choose $x_0 \not = x_\ast$, and let $U_0,V_0$ be disjoint open neighborhoods of $x_0,x_\ast$, respectively.
For the inductive step, suppose that $(x_i)_{ i \leq n}$, $(U_i)_{ i \leq n}$, $(V_i)_{ i \leq n}$ are already chosen and satisfy the following conditions:
\begin{enumerate}

\item for all $i \leq n$, we have $x_i \in U_i$, $x_\ast \in V_i$, and $U_i$, $V_i$ are open;

\item if $i<j \leq n$, then $U_i\cap U_j = \varnothing$, and

\item if $i\leq j \leq n$, then $U_i\cap V_j = \varnothing$.

\end{enumerate}
Since $x_\ast$ is a limit point, choose $x_{n+1} \in V_{n}\setminus \{x_\ast\}$. Let $U,V$ be disjoint neighborhoods of $x_{n+1}, x_\ast$, respectively, and define $U_{n+1} = U\cap V_n$, $V_{n + 1} = V \cap V _n $. It is not hard to check that the sequences $(x_i)_{ i \leq n+1}$, $(U_i)_{ i \leq n+1}$, $(V_i)_{i \leq n+1}$ satisfy all the desired properties.

Once we have constructed $(U_i)_{i < \omega}$, we let $a_0 = [X\setminus \bigcup_{1\leq n<\omega} U_n]$ and for $i>0$ we let $a_i = [U_i]$. Then $(a_i)_{i<\omega}$ is the desired sequence of nonzero pairwise orthogonal clopens of $\baire X$.
\endproof

\begin{remark}
By Lemma \ref{lemmOmDisc}, the Baire algebra of the Cantor space is $\omega$-dis\-connected.
In contrast, the Cantor space, while disconnected, is not $\omega$-disconnected because, by compactness, any partition into disjoint open sets must be finite.
More generally, no compact space can be $\omega$-disconnected.
\end{remark}

\begin{remark}\label{remKappaDis}
Let $X$ be an infinite separable Hausdorff Baire space.
Then $\baire X$ is $\omega$-disconnected by Lemma~\ref{lemmOmDisc}. On the other hand,
$\baire X$ is not $\kappa$-disconnec\-ted for any cardinal $\kappa > \omega$. For, suppose that $(a_i)_{i<\kappa} \subseteq \baire {X}$ are pairwise orthogonal and clopen. Then for each $i <\kappa $, there is an open set $A_i\subseteq X$ with $a_i = [A_i]$. We must have that $A_i \cap A_j = \varnothing$ if $i\not = j$, for otherwise their intersection, being open, would be non-meager. It follows that $\kappa \le \omega$ as no separable space admits an uncountable collection of disjoint open subsets. 
\end{remark}

\begin{theorem}\label{theoExistsEmbedding}
Let $X$ be a Baire space and $W$ an $\sf S5$-frame with the upper cluster size $\lambda$ and the number of clusters $\kappa$.
If $\baire X$ is $\kappa$-disconnected and $\lambda$-resolvable, then there is an embedding $h\colon {\Kur}(W) \to \baire X$.
\end{theorem}

\proof
Let $(C^\iota)_{\iota <\kappa}$ enumerate the clusters of $W$ and let $\lambda_\iota = |C^\iota |$. Enumerate each $C^\iota$ by $( w^\iota_\nu) _{\nu < \lambda_\iota}$.
Since $\baire X$ is $\kappa$-disconnected, there are pairwise orthogonal elements $([U^\iota])_{\iota < \kappa}$ such that each $U^\iota$ is open and $\bigsqcup_{\iota < \kappa} [U^\iota] = [X]$.
Because $X$ is a Baire space, the $U^\iota$ must be pairwise disjoint as their intersection is open and meager.
By Lemma~\ref{lemmGlobalToLocalRes}, for each $\iota < \kappa$ and $\nu < \lambda_\iota$, there are $A^\iota_\nu \subseteq U^\iota$ that Baire $\lambda_\iota$-resolve $U^\iota$.

Define a partial function $f\colon X \to W$ by $f (x) = w$ if there exist $\iota < \kappa$, $\nu < \lambda_\iota$ such that $x \in A^\iota_\nu$ and $w = w^\iota_\nu$. We show that $f$ is a Baire map.
It is Baire-continuous since if $ V \subseteq W$ is open, 
then $V$ is a union of clusters $\bigcup_{\iota \in I} C^\iota$, hence $f^{-1}(V) = \bigcup_{\iota \in I} f^{-1} (C^\iota) = \bigcup_{\iota \in I} U^{\iota}$. 
It is Baire-open since if $V\subseteq X $ is open, $M\subseteq X$ meager, and $w\in f(V\setminus M)$, then $ w = w^\iota_\nu$ for some $\iota$, $\nu$.
It follows from the definition of $f$ that $V\cap U^\iota\neq\varnothing $, and since the sets $A^\iota_{\nu'}$ yield a Baire $\lambda_\iota$-resolution of $U^\iota$, we have that $A^\iota_{\nu'}  \cap V $ is non-meager for any $\nu' < \lambda_\iota$. Therefore,
$(A^\iota_{\nu'}  \cap V)\setminus M$ is nonempty, and hence $w^\iota_{\nu'} \in f(V\setminus M) $. Since $\nu'$ was arbitrary, $C^\iota \subseteq f(V\setminus M)$, and since $w$ was arbitrary,  $f(V\setminus M)$ is open.
The map $f$ is defined almost everywhere since $[X] = \bigsqcup_{\iota<\kappa}  [U^\iota]$, so $X\setminus \bigcup_{\iota<\kappa} U^\iota$ is meager.
It is trivially meager since the only meager subset of $W$ is the empty set.
Finally, it is exact since $f^{-1}(w^\iota_\nu) = A^\iota_\nu$, which is non-meager.
Thus, $h\colon {\Kur}(W) \to \baire X$ given by $h(B) = [f^{-1} (B)]$ is an embedding by Lemma~\ref{lem:BaireIsEmbed}.
\endproof

\begin{theorem}
If $X$ is $\omega$-resolvable, then ${\sf S5U}$ is sound and strongly complete for $\baire X$.
\end{theorem}

\proof
Soundness follows from $\baire X$ being a monadic algebra. 
For completeness, if $\Gamma \not\vdash_{{\sf S5U}}  \varphi$, then by Theorem \ref{theoStrongCompCount}, there is a countable $\sf S5U$-frame $W$ and a valuation $\val\cdot_W$ such that
$\Gamma \not \models_{W}  \varphi$. By Theorem~\ref{theoExistsEmbedding},
there is an embedding $f\colon \Kur (W) \to \baire X$. Thus,
$\Gamma \not \models_{\baire X} \varphi$.
\endproof

\begin{corollary}\label{corCompU}
Let $X$ be a crowded space which is either completely metrizable and of cardinality continuum or locally compact Hausdorff. Then ${\sf S5U}$ is sound and strongly complete for $\baire X$. In particular, ${\sf S5U}$ is sound and strongly complete for $\baire {\mathbb R}$ and $\baire {\bf C}$, where $\bf C$ is the Cantor space.
\end{corollary}

Corollary~\ref{corCompU} can be viewed as an analogue of the McKinsey-Tarski theorem for ${\sf S5U}$.
As mentioned in Remark~\ref{remQ}, a similar result does not hold for $\mathbb Q$, even though $\sf S4U$ is complete for $\mathbb Q$ but not $\mathbb R$ (see, e.g., \cite{She99}), so the picture in the Baire setting is a bit different from the original topological semantics. 

\begin{remark}
Corollary~\ref{corCompU} does not extend to uncountable languages in view of Remark \ref{remKappaDis} as the collection
\[ \{\exists \Box p_\iota\}_{\iota <\kappa} \cup \{ \neg \exists (p_\iota\wedge p_\varsigma)\}_{\iota<\varsigma<\kappa} \]
cannot be satisfied on, e.g., the real line if $\kappa>\omega$ since this would require $\val {p_\iota}=[U_\iota]$ with disjoint open sets $U_\iota$, which is impossible on $\mathbb R$. 
\end{remark}

\section{Concluding remarks}

We have shown that the Baire algebra of any topological space is a closure algebra, and that the logic of these algebras is $\sf S5$, providing variants of the celebrated McKinsey-Tarski theorem for the logic $\sf S5$.
We have identified Baire resolvability as a sufficient condition for strong completeness.
For such spaces, we have also constructed subalgebras of the Baire algebra whose logic is ${\sf S5}_n$ (for any finite $n$).
Given a cardinal $\lambda$, strong completeness also holds for variants of $\sf S5$ with $\lambda$-many variables provided that the space is Baire $\lambda$-resolvable.
In addition, we have shown that if the Baire algebra is $\omega$-disconnected, then strong completeness extends to ${\sf S5U}$.
Finally, we have shown that crowded spaces which are either complete continuum-sized metric or locally compact Hausdorff enjoy the above properties, leading to a large class of concrete examples of spaces for which $\sf S5$ and $\sf S5U$ are strongly complete.

This work follows \cite{Fer10,Lan12} in studying point-free semantics for modal logic based on quotients of 
the powerset algebras of topological spaces.
There are many other $\sigma$-ideals that may be used to define similar quotients, e.g.~the $\sigma$-ideal of countable subsets.
A general and systematic study of such quotients and their associated modal logics remains an
interesting line of research.

\subsection*{Acknowledgements}

We are grateful to Jan van Mill for fruitful discussion and many useful suggestions. This included drawing our attention to the 
Disjoint Refinement Lemma---our main tool in Section~4. 
We would also like to thank the referees whose comments have considerably improved the paper.

David Fernández-Duque was supported by the FWO-FWF Lead Agency grant G030620N (FWO)/I4513N (FWF) and by the SNSF--FWO Lead Agency Grant 200021L\_196176/G0E2121N.

\def\cprime{$'$}
\providecommand{\bysame}{\leavevmode\hbox to3em{\hrulefill}\thinspace}
\providecommand{\MR}{\relax\ifhmode\unskip\space\fi MR }
\providecommand{\MRhref}[2]{%
  \href{http://www.ams.org/mathscinet-getitem?mr=#1}{#2}
}
\providecommand{\href}[2]{#2}

\bibliographystyle{plain}
\bibliography{biblio}

\end{document}